\documentclass[reqno]{amsart}

\usepackage{amsfonts}
\usepackage{amssymb}
\usepackage{amsthm}
\usepackage{setspace}
\usepackage{amsmath}
\usepackage{enumerate}
\usepackage{graphics}
\usepackage{fancybox}
\usepackage{etoolbox}
\makeatletter
\patchcmd{\section}{\scshape}{\bf}{}{}
\makeatother

\newtheorem{thm}{Theorem}[section]

\newtheorem{lemma}[thm]{Lemma}
\newtheorem{prop}[thm]{Proposition}
\theoremstyle{definition}
\newtheorem{Def}[thm]{Definition}
\newtheorem{ex}[thm]{Example}
\newtheorem{rmk}[thm]{Remark}

\makeatletter \renewenvironment{proof}[1][\proofname] {\par\pushQED{\qed}\normalfont\topsep6\p@\@plus6\p@\relax\trivlist\item[\hskip\labelsep\bfseries#1\@addpunct{.}]\ignorespaces}{\popQED\endtrivlist\@endpefalse} \makeatother
\begin{document}

\title{E\MakeLowercase{xploring stronger forms of transitivity on} $G$-\MakeLowercase{spaces}}

\let\MakeUppercase\relax
\author{Mukta Garg \and Ruchi Das}
\newcommand{\acr}{\newline\indent}
\address{(M. Garg) Department of Mathematics, University of Delhi, Delhi-110007, India}
\email{mgarg@maths.du.ac.in, mukta.garg2003@gmail.com}

\address{(R. Das) Department of Mathematics, University of Delhi, Delhi-110007, India}
\email{rdasmsu@gmail.com}

\renewcommand{\thefootnote}{\fnsymbol{footnote}}
\renewcommand{\thefootnote}{\arabic{footnote}}
\begin{abstract}
In this paper we introduce and study some stronger forms of transitivity like total transitivity, weakly mixing for maps on $G$-spaces. We obtain their relationship with the earlier defined notion of strongly mixing for maps on $G$-spaces. We also study $G$-minimal maps on $G$-spaces in detail.
\end{abstract}
\maketitle

\section{Introduction}
Dynamical properties of maps in dynamical systems have been extensively studied in recent years. They are of extreme importance in the qualitative study of dynamical systems. One of the very important and useful dynamical properties is topological transitivity. It plays an important role in the study of chaos theory and decomposition theorems. Apart from standard topological transitivity, various variants of this concept are proposed and studied. For example, total transitivity, topological mixing, minimality etc. One can refer \cite{AK, A, B, C, D, MD, K1, K2, K3} for results on these notions. While working in one dimensional topological dynamics, it is natural to try to extend results studied in a particular setting to more general settings. We show that some important facts from the topological dynamics work on much more general spaces than on metric spaces/topological spaces namely, on $G$-spaces, that is on topological spaces on which topological groups act continuously. Dynamical properties of group actions have been defined and studied in detail \cite{CA}. However, dynamical properties for maps on $G$-spaces apparently have not attracted much attention and a systematic study has not been done. The present paper is a sincere attempt in this direction. In \cite{TC} authors have defined strongly $G$-mixing map and used it to prove decomposition theorem on $G$-spaces. We study in detail stronger forms of transitivity on metric/topological $G$-spaces like total $G$-transitivity, strongly $G$-mixing, weakly $G$-mixing, $G$-minimality.

In Section 2, we introduce notions of total $G$-transitivity and weakly $G$-mixing for maps on $G$-spaces. We study their interrelations with strongly $G$-mixing maps on $G$-spaces. Observing that in general, notions of total $G$-transitivity and weakly $G$-mixing are independent, we provide conditions under which one notion implies the other. Section 3 is devoted to the study of $G$-minimal maps on $G$-spaces. Justifying that product of two $G$-minimal maps need not be $G\times G$-minimal on the product space, we give a sufficient condition under which product of two $G$-minimal maps become $G\times G$-minimal. Giving some characterizations of $G$-minimal maps, we show that a pseudoequivariant self map on a compact Hausdorff $G$-space possesses a $G$-minimal set.

We write $\mathbb{R}$ for the set of real numbers, $\mathbb{Z}$ for the set of integers and $\mathbb{N}$ for the set of positive integers. A (discrete) dynamical system is a pair $(X,f)$, where $X$ is a topological space and $f:X\rightarrow X$ is a continuous map. For $x\in X$, the \textit{$f$-orbit} of $x$ in $X$ is given by the set $O_f(x)=\{f^k(x):k\geqslant 0\}$, where $f^k$ is the $k^{th}$ iteration of $f$. A point $x\in X$ is said to be \textit{isolated} if $\{x\}$ is open in $X$. A point $x\in X$ is a \textit{periodic point} of $f$ if $f^k(x)=x$ for some $k\in\mathbb{N}$. The smallest such $k$ is called \textit{prime period} of $x$. The set of periodic points of $f$ is denoted by Per$(f)$. A map $f$ is said to be \textit{topologically transitive} (or \textit{transitive}) if for any pair of nonempty open subsets $U$, $V$ of $X$, there exists $k\geqslant 1$ such that $f^k(U)\cap V\ne\emptyset$. The facts that product of transitive maps need not be transitive and composition of transitive maps need not be transitive motivated the concepts of weakly mixing and total transitivity which are stronger than transitivity. A map $f$ is called \textit{totally transitive} if all its iterates $f^n$, $n\geqslant 1$, are transitive. A map $f$ is said to be \textit{strongly mixing} if for any pair of nonempty open subsets $U$, $V$ of $X$, there is $N\in\mathbb{N}$ such that for all $n\geqslant N$, $f^n(U)\cap V\ne\emptyset$. Also $f$ is said to be \textit{weakly mixing} if $f\times f$ is transitive. One can note that a strongly mixing map is weakly mixing but the converse is not true \cite{KP}. A subset $A$ of $X$ is said to be \textit{$+f$ invariant} if $f(A)\subseteq A$, \textit{$-f$ invariant} if $f^{-1}(A)\subseteq A$ and \textit{$f$-invariant} if $f(A)=A$. A dynamical system $(X,f)$ is said to be \textit{minimal} if every orbit in $X$ is dense in $X$; in that case we also say that $f$ itself is minimal. A subset $A$ of $X$ is said to be a \textit{minimal} set of $f$ if it is nonempty, closed, $+f$ invariant and $(A,f|_A)$ is minimal.

By a \textit{G-space} $X$, we mean a triple $(G,X,\theta)$, where $G$ is a topological group, $X$ is a topological space and $\theta : G\times X \rightarrow X$ is a continuous action of $G$ on $X$ \cite{BD}. We denote $\theta (g,x)$ by $g.x$, for $g\in G$ and $x\in X$. By a trivial action of $G$ on $X$, we mean $g.x=x$ for all $g\in G$, $x\in X$. Note that if $X$ is a $G$-space, then for any $g\in G$, $T_g:X\rightarrow X$ defined by $T_g(x)=g.x$, $x\in X$, is a homeomorphism. For $x\in X$, the \textit{$G$-orbit} of $x$ in $X$ is given by the set $G(x)=\{g.x : g\in G\}$. For a subset $A$ of $X$, we also define $G(A)=\{g.a:g\in G,a\in A\}$. If $X$, $Y$ are $G$-spaces, then a continuous map $f:X\rightarrow Y$ is said to be \textit{equivariant} if $f(g.x)=g.f(x)$ for every $g\in G$ and every $x\in X$ and \textit{pseudoequivariant} if $f(G(x))=G(f(x))$ for every $x\in X$. It is clear that every equivariant map is pseudoequivariant but the converse is not true \cite{RD}. Note that if $f$ is pseudoequivariant, then $f(G(A))=G(f(A))$ and $f^{-1}(G(A))=G(f^{-1}(A))$ for every subset $A$ of $X$. Consider the equivalence relation $\sim$ defined on $X$ by $x\sim y$ if $y=g.x$ for some $g\in G$. Then for any $x\in X$, the equivalence class of $x$ is $G(x)$. The set of all equivalence classes $G(x)$, $x\in X$, is denoted by $X/G$, endowed with quotient topology, it is called the \textit{orbit space} of $X$. The map $p:X\rightarrow X/G$ defined by $p(x)=G(x)$, $x\in X$, is called the \textit{orbit map} which is clearly continuous, onto and open. If $f:X\rightarrow X$ is pseudoequivariant, then its induced map $\bar{f}:X/G\rightarrow X/G$ defined by $\bar{f}(G(x))=G(f(x))$, $G(x)\in X/G$, is well defined. Note that $\bar{f}$ is continuous and $p\circ f=\bar{f}\circ p$.

A subset $A$ of $X$, where $X$ is a $G$-space, is said to be \textit{$G$-invariant} if $g.A\subseteq A$ for every $g\in G$. For $x\in X$, the associated \textit{$G_{f}$-orbit} of $x$ is given by the set $G_{f}(x)=G(O_f(x))=\{g.f^k(x):g\in G, k\geqslant 0\}$. Note that if $f:X\rightarrow X$ is pseudoequivariant, then $G_f(x)$ is the smallest $+f$ invariant, $G$-invariant set containing $x$. Also for a subset $A$ of $X$ and $f:X\rightarrow X$ pseudoequivariant, $G_f^+(A)=\cup_{g\in G}\cup_{k\geqslant 0}g.f^k(A)$ is the smallest $+f$ invariant, $G$-invariant set containing $A$ and $G_f^-(A)=\cup_{g\in G}\cup_{k\geqslant 0}g.f^{-k}(A)$ is the smallest $-f$ invariant, $G$-invariant set containing $A$. Also recall that a point $x\in X$ is called \textit{$G$-transitive point} of $f$ if its $G_f$-orbit, $G_f(x)$, is dense in $X$. The set of all $G$-transitive points of $f$ is denoted by $G$-$Trans_f$.

\section{Total transitivity and mixing on $G$-spaces}

Let $X$ be a $G$-space and $f:X\rightarrow X$ be continuous. Recall that the map $f$ is said to be \textit{$G$-transitive} $(GT)$ if for any pair of nonempty open subsets $U$, $V$ of $X$, there exists $g\in G$ such that the set
\begin{center}
$N_g(U,V)=\{k\in\mathbb{N}:g.f^k(U)\cap V\neq\emptyset\}$
\end{center}
is nonempty \cite{RT}.\\

The following example shows that if $f:X\rightarrow X$ is $G$-transitive, then $f^2$ need not be $G$-transitive.

\begin{ex}\label{ex1}
Consider $X=\{\pm\frac{1}{n},\pm (1-\frac{1}{n}): n\in\mathbb{N}\}$ with relative topology of $\mathbb{R}$. Define $h:X\rightarrow X$ by
\begin{center}
$h(x)=\left\{\begin{array}{ll} x &\mbox{if }\hspace{0.5mm} x\in\{-1,0,1\},\\ -x^+ &\mbox{if }\hspace{0.5mm} 0<x<1,\hspace{0.5mm} x\in X,\\ -x^- &\mbox{if } -1<x<0,\hspace{0.5mm} x\in X,\end{array}\right.$
\end{center}
where $x^+$ ($x^-$) denotes the element of $X$ immediate to the right (left) of $x$. Consider the action of the topological group $G=\{h^n:n\in\mathbb{Z}\}$ on $X$ given by $h^n.x=h^n(x)$ for every $n\in\mathbb{Z}$, every $x\in X$. Also define $f:X\rightarrow X$ by

\begin{center}
$f(x)=\left\{\begin{array}{ll} x &\mbox{if }\hspace{0.5mm} x\in\{-1,0,1\},\\ x^+ &\mbox{if }\hspace{0.5mm} x\in X\setminus\{-1,0,1\}.\end{array}\right.$
\end{center}
Then $G_f(x)=G(x)\cup G(f(x))=X\setminus\{-1,0,1\}$ for every $x\in X\setminus\{-1,0,1\}$, which is dense in $X$. Note that any open set containing $0$ contains points of the form $\pm 1/n$. Similarly, any open set containing $-1$ (or $1$) contains points of the form $-(1-1/n)$ (or $1-1/n$). Therefore $G_f$-orbit of every open set in $X$ is dense in $X$ implying that $f$ is $G$-transitive. On the other hand, if $U=\{2/3\}$ and $V=\{5/6\}$, then $h^k.(f^2)^n(U)\cap V=\emptyset$ for every $n\in\mathbb{N}$ and every $k\in\mathbb{Z}$ implying that $f^2$ is not $G$-transitive.
\end{ex}

The above example motivates the following definition of total $G$-transitivity.

\begin{Def}
Let $X$ be a $G$-space and $f:X\rightarrow X$ be continuous. Then $f$ is said to be \textit{totally $G$-transitive} if $f^n$ is $G$-transitive for every $n\geqslant 1$.
\end{Def}

One can observe that under the trivial action of $G$ on $X$, notions of total transitivity and total $G$-transitivity coincide. Under a non-trivial action of $G$ on $X$, every totally transitive map is totally $G$-transitive but the converse is not true as justified by the following example.

\begin{ex}\label{ex}
Let $S^1$ denote the unit circle in the complex plane. Consider $X=T^n=S^1\times S^1\times\dots\times S^1$ ($n$-dimensional Torus) with standard topology and topological group $G=T^m$, where $m<n$. Denoting $e^{2\pi\iota\theta}$ in $S^1$ by its argument $\theta\in[0,1]$, define the action of $G$ on $X$ by $(g_1,g_2,\dots,g_m).(\theta_1,\theta_2,\dots,\theta_m,\theta_{m+1},\dots,\theta_n)=(\theta_1+g_1,\theta_2+g_2,\dots,\theta_m+g_m,\theta_{m+1},\dots,\theta_n)$, where $(g_1,g_2,\dots,g_m)\in G$. Define $f:X\rightarrow X$ by $f(\theta_1,\theta_2,\dots,\theta_m,\theta_{m+1},\dots,\theta_n)=(\theta_1,\theta_2,\dots,\theta_m,\theta_{m+1}+\beta_{m+1},\dots,\theta_n+\beta_n)$, where $\{\beta_{m+1},\beta_{m+2},\dots,\beta_n\}$ is rationally independent (i.e. $\{\beta_{m+1},\beta_{m+2},\dots,\\\beta_n,1\}$ is linearly independent over $\mathbb{Q}$). Then we can find $h_m\in\mathbb{R}$ such that $h_m\notin\mbox{span}\{\beta_{m+1},\dots,\beta_n,1\}$ (over $\mathbb{Q}$) so that the set $\{h_m,\beta_{m+1},\dots,\beta_n,1\}$ becomes linearly independent over $\mathbb{Q}$. Continuing like this we can find $h_1,h_2,\dots,h_m$ in $\mathbb{R}$ such that $\{h_1,h_2,\dots,h_m,\beta_{m+1},\dots,\beta_n,1\}$ is linearly independent over $\mathbb{Q}$. For $(\theta_1,\theta_2,\dots,\theta_n)\in X$, $G_f(\theta_1,\theta_2,\dots,\theta_n)\supseteq O_f(\theta_1+h_1,\dots,\theta_m+h_m,\theta_{m+1},\dots,\theta_n)$, which is dense in $X$, by \cite[(1.14)]{JD}. Therefore $G_f$-orbit of every point in $X$ is dense in $X$ implying that $f$ is $G$-transitive. Similarly, $f^2$ is given by $f^2(\theta_1,\theta_2,\dots,\theta_m,\dots,\\\theta_n)=(\theta_1,\theta_2,\dots,\theta_m,\theta_{m+1}+2\beta_{m+1},\dots,\theta_n+2\beta_n)$, which is $G$-transitive since the set $\{2\beta_{m+1},\dots,2\beta_n,1\}$ is also linearly independent over $\mathbb{Q}$. Thus continuing like this we get $f^k$ is $G$-transitive for every $k\geqslant 1$ and hence $f$ is totally $G$-transitive. However, $f$ is not totally transitive. For if $n=2$ and $m=1$, then the map $f$ is given by $f(\theta_1,\theta_2)=(\theta_1,\theta_2+\beta_2)$, where $\beta_2$ is irrational. Note that for $U_1=V_1=\{\theta:1/8<\theta<1/6\}$ (i.e. open arc joining $(\cos\frac{\pi}{4},\sin\frac{\pi}{4})$ and $(\cos\frac{\pi}{3},\sin\frac{\pi}{3})$) and $U_2=V_2=\{\theta:5/8<\theta<2/3\}$, $f^k(U_1\times V_1)\cap (U_2\times V_2)=\emptyset$ for every $k\in\mathbb{N}$ implying that $f$ is not transitive.
\end{ex}

\begin{Def} \cite{TC}
Let $X$ be a $G$-space and $f:X\rightarrow X$ be continuous. Then $f$ is said to be \textit{strongly $G$-mixing} if for any pair of nonempty open subsets $U$, $V$ of $X$, there exists $N\in\mathbb{N}$ such that for all $n\geqslant N$, there is $g_n\in G$ such that
\begin{center}
$g_n.f^n(U)\cap V\ne\emptyset$.
\end{center}
\end{Def}

Note that under the trivial action of $G$ on $X$, notions of strongly $G$-mixing and strongly mixing coincide. In general, under a non-trivial action of $G$ on $X$, a strongly mixing map is strongly $G$-mixing but the converse is not true as justified by the following example.

\begin{ex}\label{mixing}
Consider $X=[-1,1]$ with relative topology of $\mathbb{R}$ and the action of additive group of integers modulo 2, $G=\mathbb{Z}_2=\{0,1\}$ with discrete topology on $X$, given by $0.x=x$, $1.x=-x$, $x\in X$. Define $f:X\rightarrow X$ by
\begin{center}
$f(x)=\left\{\begin{array}{ll} -2x-2 &\mbox{if }\hspace{0.5mm} -1\leqslant x\leqslant -1/2,\\ \hspace{0.5mm}2x &\mbox{if }\hspace{0.5mm} -1/2<x<1/2,\\ -2x+2 &\mbox{if }\hspace{1mm} 1/2\leqslant x\leqslant 1.\end{array}\right.$
\end{center}
Then one can observe that for $U=(-1/2,0)$ and $V=(0,1/2)$, $f^n(U)\cap V=\emptyset$ for every $n\in\mathbb{N}$ which implies that $f$ is not transitive and hence not strongly mixing. We show that $f$ is strongly $G$-mixing. Let $U$, $V\subseteq [-1,1]$ be nonempty open sets, $(a,b)\subseteq U\cap [0,1]$ and $(c,d)\subseteq U\cap [-1,0]$ so that at least one of $(a,b)$ or $(c,d)$ is nonempty. Suppose $(a,b)$ is nonempty. Then one can note that $f^k(a,b)=[0,1]$ for some $k\in\mathbb{N}$ so that if $V\cap [0,1]\ne\emptyset$, then $0.f^n(U)\cap V\ne\emptyset$ for all $n\geqslant k$. On the other hand if $V\subseteq [-1,0)$, then $1.f^n(U)\cap V\ne\emptyset$ for all $n\geqslant k$ and hence $f$ is strongly $G$-mixing.
\end{ex}

Recall that if $X$ is a $G$-space, then $X\times X$ is a $G\times G$-space under the action $(g,h).(x,y)=(g.x,h.y)$, for $(g,h)\in G\times G$, $(x,y)\in X\times X$.

\vspace{2mm}
Next we define the notion of weakly $G$-mixing for continuous self maps on $G$-spaces and study its relation with strongly $G$-mixing and total $G$-transitivity.

\begin{Def}
Let $X$ be a $G$-space and $f:X\rightarrow X$ be continuous. Then $f$ is said to be \textit{weakly $G$-mixing} if the Cartesian product $f\times f$ is $G\times G$-transitive, i.e. for every pair $U\times V$, $E\times F$ of nonempty basic open subsets of $X\times X$, there exist $(g,h)\in G\times G$ and $k\in\mathbb{N}$ such that
\begin{center}
$(g,h).(f\times f)^k(U\times V)\cap (E\times F)\ne\emptyset$
\end{center}
equivalently,
\begin{center}
$g.f^k(U)\cap E\ne\emptyset\hspace{4mm}$ and $\hspace{4mm}h.f^k(V)\cap F\ne\emptyset$.
\end{center}
\end{Def}

Every strongly $G$-mixing map is weakly $G$-mixing follows from the following result.

\begin{prop}
Let $X$ be a $G$-space and $f:X\rightarrow X$ be continuous. If $f$ is strongly $G$-mixing, then it is weakly $G$-mixing.
\end{prop}
\begin{proof}
Let $U\times V$, $E\times F$ be nonempty basic open subsets of $X\times X$. Since $f$ is strongly $G$-mixing, there exist $N_1$, $N_2\in\mathbb{N}$ such that for all $n\geqslant N_1$, there is $g_n\in G$ such that $g_n.f^n(U)\cap E\ne\emptyset$ and for all $m\geqslant N_2$, there is $h_m\in G$ such that $h_m.f^m(V)\cap F\ne\emptyset$. Choosing $N=\max\{N_1,N_2\}$ we get the required result.
\end{proof}

Note that under the trivial action of $G$ on $X$, the notion of weakly $G$-mixing coincides with that of weakly mixing. In general, under a non-trivial action of $G$ on $X$, a weakly mixing map is weakly $G$-mixing but the converse is not true as shown in the following example.

\begin{ex}
Consider $G$, $X$ and $f$ as given in Example \ref{mixing}. Note that the map $f$ being strongly $G$-mixing is weakly $G$-mixing. However, for $U\times V=(0,1/2)\times (0,1/2)$ and $E\times F=(-1/2,0)\times (-1/2,0)$, $(f\times f)^k(U\times V)\cap (E\times F)=\emptyset$ for every $k\in\mathbb{N}$ implying that $f$ is not weakly mixing.
\end{ex}

The following result shows that every strongly $G$-mixing map is totally $G$-transitive.

\begin{prop}
Let $X$ be a $G$-space and $f:X\rightarrow X$ be continuous. If $f$ is strongly $G$-mixing, then it is totally $G$-transitive.
\end{prop}
\begin{proof}
Let $m\in\mathbb{N}$ and $U$, $V$ be nonempty open subsets of $X$. Since $f$ is strongly $G$-mixing, there exists $N\in\mathbb{N}$ such that for all $n\geqslant N$, there is $g_n\in G$ such that $g_n.f^n(U)\cap V\ne\emptyset$. Let $k$ be the smallest multiple of $m$ greater than $N$. Then $g_k.(f^m)^{k/m}(U)\cap V\ne\emptyset$, $g_k\in G$ which proves that $f^m$ is $G$-transitive.
\end{proof}

Next we show that every pseudoequivariant weakly $G$-mixing map is totally $G$-transitive. We first prove the following useful lemma.

\begin{lemma}\label{l1}
Let $X$ be a $G$-space and $f:X\rightarrow X$ be pseudoequivariant. If $f$ is weakly $G$-mixing, then $f\times f\times\dots\times f$ ($n$-times) is $G\times G\times\dots\times G$ ($n$-times) transitive for every $n\in\mathbb{N}$.
\end{lemma}
\begin{proof}
We first show that $G\times G$-transitivity of $f\times f$ implies $G$-transitivity of $f$. Let $U$, $V$ be nonempty open subsets of $X$. Since $f$ is weakly $G$-mixing, there exist $(g,h)\in G\times G$ and $k\in\mathbb{N}$ such that $(g,h).(f\times f)^k(U\times X)\cap (V\times X)\ne\emptyset$ so that $g.f^k(U)\cap V\ne\emptyset$. Therefore for every pair of nonempty open subsets $U$, $V$ of $X$, there exists $g\in G$ such that $N_g(U,V)\ne\emptyset$. Now let $U_1$, $U_2$, $V_1$, $V_2$ be nonempty open subsets of $X$. Again $f$ being weakly $G$-mixing, there exist $(g,h)\in G\times G$ and $k\in\mathbb{N}$ such that $(g,h).(f\times f)^k(U_1\times V_1)\cap (U_2\times V_2)\ne\emptyset$ so that $g.f^k(U_1)\cap U_2\ne\emptyset$ and $h.f^k(V_1)\cap V_2\ne\emptyset$. Then $A=U_1\cap g'.f^{-k}(U_2)$ and $B=V_1\cap h'.f^{-k}(V_2)$ are nonempty open subsets of $X$ for some $g'$, $h'\in G$ (using pseudoequivariancy of $f$) which gives $g_0\in G$ such that $N_{g_0}(A,B)\ne\emptyset$. Let $n\in N_{g_0}(A,B)$, it follows that $g_0.f^n(U_1)\cap V_1\ne\emptyset$ and $g_0'.f^n(U_2)\cap V_2\ne\emptyset$ for some $g_0'\in G$ implying that $N_{g_0}(U_1,V_1)\cap N_{g_0'}(U_2,V_2)\ne\emptyset$. Thus by induction we have $\cap_{i=1}^nN_{g_i}(U_i,V_i)\ne\emptyset$ for some $g_1,g_2,\dots,g_n\in G$ which in turn implies that for any finite collection of nonempty open subsets $U_1$, $U_2,\dots, U_n$, $V_1$, $V_2,\dots, V_n$ of $X$, there exists $(g_1,g_2,\dots,g_n)\in G\times G\times\dots\times G$ such that $N_{(g_1,g_2,\dots,g_n)}(U_1\times U_2\times\dots\times U_n,V_1\times V_2\times\dots\times V_n)=\cap_{i=1}^nN_{g_i}(U_i,V_i)\ne\emptyset$ for all $n\geqslant 1$. Hence the required result follows.
\end{proof}

\begin{prop}
Let $X$ be a $G$-space and $f:X\rightarrow X$ be pseudoequivariant. If $f$ is weakly $G$-mixing, then $f$ is totally $G$-transitive.
\end{prop}

\begin{proof}
Suppose that $f^m$ is not $G$-transitive for some $m>1$. Then there exists a subset $F$ of $X$, which is nonempty, proper, closed, $G$-invariant, $+f^m$ invariant and hence $+f^{mn}$ invariant for any $n\geqslant 1$ such that int$(F)\ne\emptyset$. This implies that $f^{mn}$ is not $G$-transitive for any $n\geqslant 1$. Therefore for any given $n\geqslant 1$, there exist nonempty open subsets $U_n$, $V_n$ of $X$ such that for every $g\in G$ and every $p\geqslant 1$ we have $g.(f^{mn})^p(U_n)\cap V_n=\emptyset$. Note that same $U_1$, $V_1$ will work for all $n$, so without loss of generality we can assume that $U$, $V$ are nonempty open subsets of $X$ such that $g.f^{mk}(U)\cap V=\emptyset$ for every $g\in G$ and every $k\geqslant 1$. Since $f$ is pseudoequivariant, $U\cap g.f^{-mk}(V)=\emptyset$ for every $g\in G$ and every $k\geqslant 1$. We claim that $f\times f\times\dots\times f$ ($m$-times) is not $G\times G\times\dots\times G$ ($m$-times) transitive. Consider the sets $V'=V\times f^{-1}(V)\times\dots\times f^{-(m-1)}(V)$ and $U'=U\times U\times\dots\times U$. Then $U'\cap (g_1,g_2,\dots,g_m).(f\times f\times\dots\times f)^{-r}(V')=\emptyset$ for every $(g_1,g_2,\dots,g_m)\in G\times G\times\dots\times G$ and every $r\geqslant 1$ which gives a contradiction to the Lemma \ref{l1}. Thus $f^m$ is $G$-transitive for every $m\geqslant 1$.
\end{proof}

\begin{rmk}
A totally $G$-transitive map need not be weakly $G$-mixing as illustrated in the following example.
\end{rmk}

\begin{ex}\label{imp}
Consider the action of $G=\mathbb{Z}_2$ on $S^1$ with standard topology given by $0.\theta=\theta$, $1.\theta=-\theta$, $\theta\in S^1$ and irrational rotation on $S^1$ given by $f(\theta)=\theta+\alpha$. Then $f$ is totally $G$-transitive. However, $f$ is not weakly $G$-mixing, for this take open sets $U=\{\theta:1/12<\theta<1/8\}$, $V_1=\{\theta:1/6<\theta<1/4\}$, $V_2=\{\theta:5/12<\theta<1/2\}$ of $S^1$ and the basic open subsets $U\times U$ and $V_1\times V_2$ of $S^1\times S^1$. Now if $0.f^{n_1}(U)\cap V_1\ne\emptyset$ for some $n_1\in\mathbb{N}$, then $f$ being an isometry, $0.f^{n_1}(U)\cap V_2=\emptyset$ and $1.f^{n_1}(U)\cap V_2=\emptyset$. Similarly, if $1.f^{n_2}(U)\cap V_1\ne\emptyset$ for some $n_2\in\mathbb{N}$, then $g.f^{n_2}(U)\cap V_2=\emptyset$ for every $g\in\mathbb{Z}_2$. Thus $f$ is not weakly $G$-mixing.
\end{ex}

\begin{Def}\cite{TC}
Let $X$ be a $G$-space and $f:X\rightarrow X$ be continuous. Then $x\in X$ is said to be \textit{$G_f$-periodic point} of $f$ if there exist $g\in G$ and $k\in\mathbb{N}$ such that $g.f^k(x)=x$. The smallest such $k$ is called \textit{$G_f$-prime period} of $x$.
\end{Def}

\begin{rmk}
Note that every periodic point of a self map $f$ on a $G$-space $X$ is a $G_f$-periodic point of $f$ which implies that if the set of periodic points of $f$ is dense in $X$, then the set of $G_f$-periodic points of $f$ is also dense in $X$. However, in Example \ref{ex1}, every point is a $G_f$-periodic point of $f$ but $Per(f)=\{-1,0,1\}$.
\end{rmk}

Next result gives a sufficient condition for a totally $G$-transitive map to be weakly $G$-mixing.

\begin{prop}\label{p3}
Let $X$ be a $G$-space and $f:X\rightarrow X$ be pseudoequivariant and totally $G$-transitive with dense set of $G_f$-periodic points. Then $f$ is weakly $G$-mixing.
\end{prop}
\begin{proof}
Let $U\times V$, $E\times F$ be nonempty basic open subsets of $X\times X$. Since $f$ is $G$-transitive, there exist $g_1\in G$ and $k\in\mathbb{N}$ such that $g_1.f^k(U)\cap E\ne\emptyset$ which implies that the set $W=U\cap f^{-k}(g_1^{-1}.E)$ is open and nonempty. Then the set of $G_f$-periodic points being dense, there exists a $G_f$-periodic point $x$ in $W$, say of $G_f$-prime period $m$, such that $g_0.f^m(x)=x$ for some $g_0\in G$. Now since $f^{-k}(F)$ is open, nonempty and $f^m$ is $G$-transitive, there exist $g_2\in G$ and $j\in\mathbb{N}$ such that $g_2.f^{mj}(V)\cap f^{-k}(F)\ne\emptyset$. Then $f$ being pseudoequivariant, $h.f^{mj+k}(V)\cap F\ne\emptyset$ for some $h\in G$. Now using pseudoequivariancy of $f$ and $G_f$-periodicity of $x$ repeatedly we get $f^{mj}(x)=h_0.x$ for some $h_0\in G$ which in turn gives $g.f^{mj+k}(x)=g_1.f^k(x)\in E$ for some $g\in G$. Thus $g.f^{mj+k}(U)\cap E\ne\emptyset$ and hence $f$ is weakly $G$-mixing.
\end{proof}

\begin{rmk}
Note that the Example \ref{imp} justifies that in general, a totally $G$-transitive map need not be strongly $G$-mixing.
\end{rmk}

Recall that a topological space is said to be \textit{second countable} if it has a countable base and \textit{non-meager} if it is not the union of a countable family of nowhere dense subsets.

The following result shows that under certain conditions $G$-transitivity implies strongly $G$-mixing.

\begin{prop}
Let  $X$ be a second countable, non-meager $G$-space and $f:X\rightarrow X$ be pseudoequivariant and $G$-transitive with $G_f(x)$ dense in $X$ for some $x\in X$. If for each neighbourhood $W$ of $x$, there exists $N\in\mathbb{N}$ such that for all $n\geqslant N$, there is $g_n\in G$ such that $g_n.f^n(W)\cap W\ne\emptyset$, then $f$ is strongly $G$-mixing.
\end{prop}

\begin{proof}
Let $U$, $V$ be nonempty open subsets of $X$. Then there exist $g_1$, $g_2\in G$ and $k_1$, $k_2\geqslant 0$ such that $g_1.f^{k_1}(x)\in U$ and $g_2.f^{k_2}(x)\in V$ implying that $x\in h_1.f^{-k_1}(U)\cap h_2.f^{-k_2}(V)=W$ (say) for some $h_1$, $h_2\in G$. Since $W$ is an open neighbourhood of $x$, there exists $N\in\mathbb{N}$ such that for all $n\geqslant N$, there is $g_n\in G$ such that $g_n.f^n(W)\cap W\ne\emptyset$. This gives $f^{k_2}(g_n.f^n(h_1.f^{-k_1}(U))\cap h_2.f^{-k_2}(V))\ne\emptyset$ which in turn implies that for all $n\geqslant N$, there is $h_n\in G$ such that $h_n.f^{n+k_2-k_1}(U)\cap V\ne\emptyset$. Hence $f$ is strongly $G$-mixing.
\end{proof}

\section{Minimality on $G$-spaces}

\begin{Def} \cite{T}
Let $X$ be a $G$-space and $f:X\rightarrow X$ be continuous. Then a nonempty, closed, $+f$ invariant, $G$-invariant subset $Y$ of $X$ is said to be a \textit{$G$-minimal} set of $f$ if $\overline{G_f(y)}=Y$ for every $y\in Y$. The map $f$ is said to be \textit{$G$-minimal} if $X$ itself is a $G$-minimal set.
\end{Def}

Note that under the trivial action of $G$ on $X$, concepts of minimality and $G$-minimality coincide. In general, under a non trivial action of $G$ on $X$, a minimal map is $G$-minimal but the converse is not true (refer Example \ref{ex}).

\begin{rmk}\label{c1}
Let $X$ be a $G$-space and $f:X\rightarrow X$ be continuous. Then one can observe that
\begin{enumerate}[$(a)$]
\item if $f$ is pseudoequivariant, then $f$ is $G$-minimal iff $X$ does not contain any nonempty, proper, closed, $+f$ invariant, $G$-invariant subset.
\item if $f$ is pseudoequivariant and $G$-minimal, then $f(X)$ is dense in $X$. If additionally, $X$ is compact and Hausdorff, then $f$ is onto.
\item if $f$ is $G$-minimal and $Y$ is a $+f$ invariant, $G$-invariant subset of $X$, then $f|_Y$ is also $G$-minimal.
\end{enumerate}
\end{rmk}

\begin{rmk}
One can observe that if $f\times h$ is $G\times G$-minimal, then $f$ and $h$ are $G$-minimal. Following example shows that the converse is not true.
\end{rmk}

\begin{ex}
Let $X=S^1\times S^1$ with standard topology and topological group $G=S^1$. Consider the action of $G$ on $X$ given by $g.(\theta_1,\theta_2)=(\theta_1+g,\theta_2)$, $g\in G$, $(\theta_1,\theta_2)\in X$. Define $f:X\rightarrow X$ by $f(\theta_1,\theta_2)=(\theta_1,\theta_2+\pi/4)$. Then $f$ is $G$-minimal. However, $(G\times G)_{f\times f}((0,0),(0,0))$ is not dense in $X\times X$ since $U_1\times U_2\times U_3\times U_4$, where $U_1=\{\theta:5/12<\theta<7/12\}$, $U_2=\{\theta:1/12<\theta<1/6\}$, $U_3=U_4=\{\theta:11/12<\theta\leqslant 1\}\cup\{\theta:0\leqslant\theta<1/12\}$, is an open set in $X\times X$ containing $(\pi.f(0,0),(0,0))$ but not intersecting $(G\times G)_{f\times f}((0,0),(0,0))$.
\end{ex}

Following result gives a sufficient condition for the product of two $G$-minimal maps to be $G\times G$-minimal on the product space.

\begin{prop}
Let $X$, $Y$ be $G$-spaces and $f:X\rightarrow X$, $h:Y\rightarrow Y$ be pseudoequivariant, $G$-minimal maps. Then $f\times h$ is $G\times G$-minimal iff for all $g$, $k\in G$, $x\in X$, $y\in Y$, $(g.f(x),y)$, $(x,k.h(y))\in\overline{(G\times G)_{f\times h}(x,y)}$.
\end{prop}
\begin{proof}
Let $G'=G\times G$. First we claim that $\overline{G'_{f\times h}(x,y)}=\overline{G_f(x)}\times\overline{G_h(y)}$ for all $x\in X$, $y\in Y$ iff $(g.f(x),y)$, $(x,k.h(y))\in\overline{G'_{f\times h}(x,y)}$ for all $g$, $k\in G$, $x\in X$, $y\in Y$. Let $(g.f(x),y)$, $(x,k.h(y))\in\overline{G'_{f\times h}(x,y)}$ for all $g$, $k\in G$, $x\in X$, $y\in Y$. Clearly $\overline{G'_{f\times h}(x,y)}\subseteq\overline{G_f(x)}\times\overline{G_h(y)}$ for all $x\in X$, $y\in Y$. Also by hypothesis of the claim and using induction, one can prove that $(g.f^m(x),k.h^n(y))\in\overline{G'_{f\times h}(x,y)}$ for all $g$, $k\in G$, $x\in X$, $y\in Y$, $m\geqslant 0$, $n\geqslant 0$. This in turn implies that $\overline{G_f(x)}\times\overline{G_h(y)}\subseteq\overline{G'_{f\times h}(x,y)}$ for all $x\in X$, $y\in Y$. Converse is straightforward. Hence the claim.

Now if $(g.f(x),y)$, $(x,k.h(y))\in\overline{G'_{f\times h}(x,y)}$ for all $g$, $k\in G$, $x\in X$, $y\in Y$, then by the claim and using $G$-minimality of both $f$ and $h$ we have $\overline{G'_{f\times h}(x,y)}=\overline{G_f(x)}\times\overline{G_h(y)}=X\times Y$ for all $x\in X$, $y\in Y$. Conversely, by $G$-minimality of both $f$, $h$ and $G\times G$-minimality of $f\times h$ we have $\overline{G'_{f\times h}(x,y)}=\overline{G_f(x)}\times\overline{G_h(y)}$ for all $x\in X$, $y\in Y$ and hence by the claim we have $(g.f(x),y)$, $(x,k.h(y))\in\overline{G'_{f\times h}(x,y)}$ for all $g$, $k\in G$, $x\in X$, $y\in Y$.
\end{proof}

Note that if $x\in X$ is a $G_f$-periodic point of $f$ with $G_f$-prime period $k$ and $f:X\rightarrow X$ is pseudoequivariant, then $G_f(x)=\cup_{m=0}^{k-1}G(f^m(x))$.

\begin{prop}
Let $X$ be a Hausdorff $G$-space, where $G$ is a compact group and $f:X\rightarrow X$ be pseudoequivariant and $G$-minimal. Then either $X$ has no isolated points or $X$ is a single $G_f$-orbit.
\end{prop}

\begin{proof}
If $X$ has no isolated points, we are done. If $x\in X$ is an isolated point, then $f$ being $G$-minimal, $\overline{G_f(f(x))}=X$ which implies that $x=g.f^k(x)$ for some $g\in G$ and $k\geqslant 1$. Therefore $G_f(x)=\cup_{m=0}^{k-1}G(f^m(x))$. Now $G$ being compact and $X$ being Hausdorff, $G(y)$ is closed in $X$ for every $y\in X$. Thus $X=\overline{G_f(x)}=G_f(x)$.
\end{proof}
\begin{prop}
Let $X$ be a $G$-space and $f:X\rightarrow X$ be pseudoequivariant and $G$-transitive ($GT$). If $M\subseteq X$ is a $G$-minimal set of $f$, then either $M=X$ or $M$ is nowhere dense in $X$.
\end{prop}

\begin{proof}
If $M=X$, we are done. Suppose that $M\ne X$. Since $M$ is a $G$-minimal set, it is nonempty, closed, $+f$ invariant, $G$-invariant. Then $X\setminus M$ is $-f$ invariant and $G$-invariant so that $G_f^-(X\setminus M)=X\setminus M$. Since $f$ is $G$-transitive, pseudoequivariant and $X\setminus M$ is a nonempty open set, $\overline{G_f^-(X\setminus M)}=X$ implying that $\overline{X\setminus M}=X$. Thus int$(M)=\emptyset$.
\end{proof}
Next result shows that a pseudoequivariant map on a $G$-space is $G$-minimal iff its induced map on the related orbit space is minimal.

\begin{prop}
Let $X$ be a $G$-space and $f:X\rightarrow X$ be pseudoequivariant. Then $f$ is $G$-minimal iff its induced map $\bar{f}:X/G\rightarrow X/G$ is minimal.
\end{prop}
\begin{proof}
Suppose $f$ is $G$-minimal. Let $G(x)\in X/G$ and $U$ be a nonempty open subset of $X/G$. Then $p^{-1}(U)$ is a nonempty open subset of $X$ so that there exist $g\in G$, $k\geqslant 0$ such that $g.f^k(x)\in p^{-1}(U)$. This gives $\bar{f}^k(G(x))\in U$ and hence $O_{\bar{f}}(G(x))$ is dense in $X/G$.

Conversely, suppose $\bar{f}$ is minimal. Let $x\in X$ and $U$ be a nonempty open subset of $X$. Then $p(U)$ is a nonempty open subset of $X/G$ so that there exists $k\geqslant 0$ such that $\bar{f}^k(G(x))\in p(U)$. This implies that there exists $g\in G$ such that $g.f^k(x)\in U$ and hence $\overline{G_f(x)}=X$.
\end{proof}

We now obtain a nice characterization of pseudoequivariant $G$-minimal maps in terms of open sets in a sequentially compact $G$-space.

\begin{prop}
Let $X$ be a sequentially compact $G$-space and $f:X\rightarrow X$ be pseudoequivariant. Then $f$ is $G$-minimal iff for every nonempty open subset $U$ of $X$, there exists $n\in\mathbb{N}$ such that $\cup_{g\in G}\cup_{k=0}^ng.f^{-k}(U)=X$.
\end{prop}
\begin{proof}
Let $f$ be $G$-minimal. Assume that there is a nonempty open subset $U$ of $X$ satisfying that for every $n\in\mathbb{N}$, there exists $x_n\in X$ such that $x_n\notin\cup_{g\in G}\cup_{k=0}^ng.f^{-k}(U)$. Then $X$ being sequentially compact, there exists a convergent subsequence $(x_{n_k})$ of $(x_n)$ such that $x_{n_k}\rightarrow x_0$ as $k\rightarrow\infty$. Since $f$ is $G$-minimal, $\overline{G_f(x_0)}=X$ so that there exist $g\in G$, $m\geqslant 0$ such that $g.f^m(x_0)\in U$. By pseudoequivariancy of $f$, $x_0\in g'.f^{-m}(U)$ for some $g'\in G$. Now $g'.f^{-m}(U)$ being an open neighbourhood of $x_0$, there exists $k_0\in\mathbb{N}$ such that $x_{n_k}\in g'.f^{-m}(U)$ for all $k\geqslant k_0$. Therefore there exists $k\in\mathbb{N}$ sufficiently large such that $n_k>m$ and $x_{n_k}\in g'.f^{-m}(U)$. Thus $x_{n_k}\in\cup_{g\in G}\cup_{i=0}^{n_k}g.f^{-i}(U)$, which contradicts the choice of $x_{n_k}$.

Conversely, let $x\in X$. To prove $\overline{G_f(x)}=X$, let $U$ be a nonempty open subset of $X$. By hypothesis, there exists $n\in\mathbb{N}$ such that $\cup_{g\in G}\cup_{k=0}^ng.f^{-k}(U)=X$. Since $x\in X$, there exist $g\in G$, $0\leqslant k_0\leqslant n$ such that $x\in g.f^{-k_0}(U)$ which gives $g'.f^{k_0}(x)\in U$ for some $g'\in G$ and hence $\overline{G_f(x)}=X$.
\end{proof}

The following result shows that in any compact Hausdorff $G$-space there are $G$-minimal sets and any two $G$-minimal sets are either disjoint or equal.

\begin{prop}
Let $X$ be a compact Hausdorff $G$-space and $f:X\rightarrow X$ be pseudoequivariant. Then $X$ contains a nonempty, $G$-minimal, $f$-invariant subset. Also, any two distinct $G$-minimal sets of $f$ are disjoint.
\end{prop}
\begin{proof}
If $X$ itself is $G$-minimal, we are done. Let us suppose that $X$ is not $G$-minimal. Then by Remark \ref{c1}$(a)$, the collection of nonempty, proper, closed, $+f$ invariant, $G$-invariant subsets of $X$, say $\mathcal{C}$, is nonempty. Let $\{C_n:n\in\mathbb{N}\}$ be a nested sequence in $\mathcal{C}$ and $C=\cap_{n\in\mathbb{N}}C_n$. Then $C$ is closed, proper, $+f$ invariant, $G$-invariant. By compactness of $X$, it is nonempty so that $C\in\mathcal{C}$. Therefore by Zorn's lemma, $\mathcal{C}$ has a minimum element, say $A$. Using compactness and Hausdorffness of $X$ we have $f(A)\in\mathcal{C}$. Thus by minimality of $A$, $f(A)=A$ and hence $A$ is $f$-invariant. Also $A$ is $G$-minimal since it does not contain any nonempty, proper, closed, $+f$ invariant, $G$-invariant subset.

If $A$ and $B$ are two distinct $G$-minimal sets of $f$ and $x\in A\cap B$, then $f$ being pseudoequivariant, $\overline{G_f(x)}$ is a nonempty, closed, $+f$ invariant, $G$-invariant subset of both $A$ and $B$ giving $A=\overline{G_f(x)}=B$.
\end{proof}

\begin{rmk}
Note that in general, strongly $G$-mixing and $G$-minimality are not related as justified by the following examples.
\end{rmk}

\begin{ex}\label{ex2}
Consider the action of $G=\mathbb{Z}_2$ on $S^1$ given by $0.\theta=\theta$, $1.\theta=-\theta$, $\theta\in S^1$ and the doubling map $f:S^1\rightarrow S^1$ defined by $f(\theta)=2\theta$. Then $f$ is strongly $G$-mixing but not $G$-minimal.
\end{ex}

\begin{ex}
Consider the $G$-space and the map $f$ given in Example \ref{imp}. Then one can observe that $f$ is $G$-minimal but not strongly $G$-mixing.
\end{ex}

Combining above results we have the following implications, where the preconditions P1 and P2 are as follows:
\begin{enumerate}
\item[] P1: $X$ is a $G$-space, $f:X\rightarrow X$ is a pseudoequivariant map.
\item[] P2: $X$ is a $G$-space, $f:X\rightarrow X$ is a pseudoequivariant map with dense set of $G_f$-periodic points in $X$.
\end{enumerate}

Here $SGM$, $WGM$, $TGT$, $GT$ and $GM$ stand for strongly $G$-mixing, weakly $G$-mixing, totally $G$-transitive, $G$-transitive and $G$-minimal respectively.

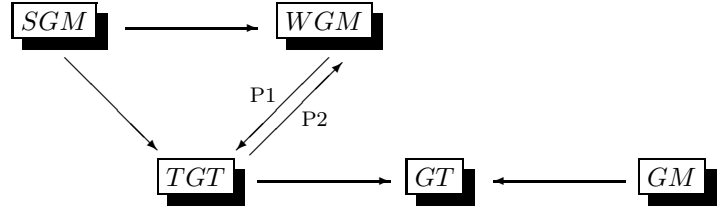
\begin{picture}(160,120)(6,-95)
\put(40,-3){\shadowbox{$SGM$}}
\put(83,5){\vector(1,0){50}}

\put(140,-3){\shadowbox{$WGM$}}
\put(160,-6){\vector(-1,-1){35}}
\put(130,-23){$\footnotesize{\mbox{P1}}$}

\put(95,-62){\shadowbox{$TGT$}}
\put(130,-43){\vector(1,1){35}}
\put(149.5,-32){$\footnotesize{\mbox{P2}}$}

\put(60,-6){\vector(1,-1){35}}

\put(133,-53){\vector(1,0){50}}
\put(189,-62){\shadowbox{$GT$}}

\put(272,-53){\vector(-1,0){50}}
\put(277,-62){\shadowbox{$GM$}}
\end{picture}

Note that $GT\not\longrightarrow TGT$ (Example \ref{ex1})

$TGT\not\longrightarrow SGM$ (Example \ref{imp})

$GT\not\longrightarrow GM$ (Example \ref{mixing})\\

We are looking for conditions under which $G$-minimality and $G$-mixing are related and also for examples justifying that weakly $G$-mixing need not imply strongly $G$-mixing.\\\\

\let\MakeUppercase\relax

\end{document}